\documentclass[12pt,reqno]{amsart}

\usepackage{color}
\usepackage{amsmath, amsthm, amssymb, amsxtra}
\usepackage{amsfonts}
\usepackage[utf8]{inputenc}
\usepackage[dvips]{epsfig}
\usepackage{graphicx}
\usepackage[english]{babel}
\usepackage{hyperref}

\usepackage{graphicx}
\usepackage{latexsym}
\usepackage{mathtools}
\usepackage{esint}
\usepackage{float}

\theoremstyle{plain}
\newtheorem{thm}{Theorem}[section]

\newtheorem{lem}[thm]{Lemma}
\newtheorem{prop}[thm]{Proposition}

\theoremstyle{definition}
\newtheorem{defn}[thm]{Definition}

\theoremstyle{remark}
\newtheorem{obs}[thm]{Remark}

\numberwithin{equation}{section}

\newcommand{\average}{{\mathchoice {\kern1ex\vcenter{\hrule height.4pt
				width 6pt depth0pt} \kern-9.7pt} {\kern1ex\vcenter{\hrule
				height.4pt width 4.3pt depth0pt} \kern-7pt} {} {} }}

\newcommand{\R}{\mathbb R}

\renewcommand{\L}{\mathcal L}

\newcommand{\p}{\partial}

\newcommand{\comment}[1]{}

\begin{document}
	
	\title[Boundary estimates for non-divergence equations in \texorpdfstring{$C^1$}{C1} domains]{Boundary estimates for non-divergence equations in \texorpdfstring{$C^1$}{C1} domains}
	\author{Clara Torres-Latorre}
	\address{Instituto de Ciencias Matemáticas \newline \indent Consejo Superior de Investigaciones Científicas\newline \indent 
		C/ Nicolás Cabrera, 13-15, 28049 Madrid, Spain.}
	\email{\tt clara.torres@icmat.es}
	
	\begin{abstract}    
    We obtain boundary nondegeneracy and regularity estimates for solutions to non-divergence equations in $C^1$ domains, providing an explicit modulus of continuity. Our results extend the classical Hopf-Oleinik lemma and boundary Lipschitz regularity for domains with $C^{1,\mathrm{Dini}}$ boundaries, while also recovering the known $C^{1-\varepsilon}$ regularity for flat Lipschitz domains, unifying both theories with a single proof.
	\end{abstract}
	
	
	\maketitle
	
	\section{Introduction}\label{sect:intro}
	Understanding the qualitative boundary behaviour of harmonic functions has been a central question in the theory of partial differential equations since as long ago as 1888, when Neumann proved that the Green function for the Laplacian has a positive normal derivative at the boundary\footnote{The original proof concerns $C^2$ planar convex domains, see \cite{AN22} for a historical review.}.
	
	Since then, research has followed two highly successful avenues. First, the quest for the minimal necessary assumptions on the domain to be able to impose Dirichlet boundary conditions has been a powerful driver of progress in potential theory and geometric measure theory. A key milestone in this area is the Wiener criterion \cite{Wie24}.
	
	On the other hand, the elliptic PDE community has been focused in more regular settings where the boundary point lemma and Lipschitz regularity at the boundary hold. The standard method involves flattening the boundary and applying regularity results for operators with variable coefficients (see, for example, \cite[Theorem 8.33]{GT98}). Moreover, optimal geometric conditions for the domain have been established, with $\p\Omega$ falling into the $C^{1,\mathrm{Dini}}$ class \cite{AN16}.
	
	However, the necessary and sufficient conditions to ensure having continuous, Lipschitz, or $C^1$ solutions are technical, and they can be difficult to verify. This raises a natural and more practical question:
	
	\vspace{0.1cm}
	\textit{Given a harmonic function in a $C^1$\! domain, what can we know about its behaviour near the boundary?}
	\vspace{0.1cm}
	
	This question was first addressed in the more general context of divergence form equations in flat Lipschitz domains in a series of works by Kozlov and Maz'ya, see for instance \cite{KM03, KM05} and the survey \cite{MS11}. Their approach was based on a change of variables that transformed understanding boundary value problems for even order elliptic equations into asymptotic analysis for evolution equations.
	
	From a purely non-divergence perspective, for operators with bounded measurable coefficients, Safonov established the boundary point lemma and boundary Lipschitz regularity in sufficiently regular domains \cite{Saf08}. While Safonov's proofs can be adapted to derive estimates for solutions near $C^1$ boundaries, the results we present are not addressed in his work.
	
	The purpose of this paper is twofold. On the one hand, we prove new quantitative growth and boundary regularity estimates in $C^1$ domains for solutions to non-divergence form second order elliptic equations with bounded measurable coefficients. On the other hand, our method is short, simple, and based on a novel approach that we believe can be of interest for future research.
	
	\subsection{Main results}
	In the following, $\mathcal{L}$ will denote a non-divergence form elliptic operator with bounded measurable coefficients,
	\begin{equation}\label{eq:non-divergence_operator}
		\mathcal{L}u = \sum\limits_{i,j = 1}^na_{ij}(x)\p^2_{ij}u, \quad \lambda I \leq A(x) \leq \Lambda I,
	\end{equation}
	where $0 < \lambda \leq \Lambda$.
	
	Our first main result is a nondegeneracy property that becomes the Hopf-Oleinik lemma when the $C^1$ modulus of continuity of the domain is Dini.
	\begin{thm}\label{thm:hopf}
		Let $\L$ be as in \eqref{eq:non-divergence_operator}, let $\Omega$ satisfy the interior $C^1$ condition at $0$ with modulus $\omega$ in the sense of Definition \ref{defn:interiorC1}, and let $u$ be a nonnegative solution to $\L u = 0$ in $\Omega$. Then, for every $0 < \rho < r < r_0$,
		$$\frac{u(\rho e_n)}{\rho} \geq \frac{1}{C}\frac{u(r e_n)}{r}\exp \left(-C\int_\rho^{2r}\omega(s)\frac{\mathrm{d}s}{s}\right),$$
		where $C$ and $r_0$ are positive and depend only on $\omega$, the dimension, and ellipticity constants.
	\end{thm}
	
	On the other hand, having an exterior $C^1$ condition ensures a bound on the growth of $\L$-harmonic functions at the boundary, that becomes linear when the $C^1$ modulus is Dini. The regularity up to the boundary follows by standard techniques.
	\begin{thm}\label{thm:upper}
		Let $\L$ be as in \eqref{eq:non-divergence_operator}, let $\Omega$ satisfy the exterior $C^1$ condition at $0$ with modulus $\omega$ in the sense of Definition \ref{defn:exteriorC1}, and let $u$ be a solution to
		$$\left\{\begin{array}{rclll}
			\L u & = & f & \text{in} & \Omega\cap B_1\\
			u & = & g & \text{on} & \p\Omega\cap B_1,
		\end{array}\right.$$
		where $f \in L^n$, $\omega_f(s) := \sup\limits_{x \in \Omega\cap B_1}\|f\|_{L^n(B_s(x))}$, and $g \in C^{1,\omega_g}$. 
		
		Then, for every $0 < \rho < r < r_0$,
		$$\frac{\|v\|_{L^\infty(B_\rho)}}{\rho} \leq C\left(\frac{\|v\|_{L^\infty(B_r)}}{r} + \int_\rho^{2r}[\omega_f+\omega_g](s)\frac{\mathrm{d}s}{s}\right)\exp \left(C\int_\rho^{2r}\omega(s)\frac{\mathrm{d}s}{s}\right),$$
		where $v = u - g(0) - \nabla g (0)\cdot x'$.
		
		\vspace{0.1cm}
		Moreover, if $\L$ has Dini continuous coefficients, $u \in C^{\tilde\omega}(\overline{\Omega}\cap B_{1/2})$, where
		$$\tilde\omega(t) = Ct\left(\|u\|_{L^\infty(B_1)}+\int_t^{2r_0}[\omega_f+\omega_g](s)\frac{\mathrm{d}s}{s}\right)\exp \left(C\int_t^{2r_0}\omega(s)\frac{\mathrm{d}s}{s}\right)$$
		is a modulus of continuity. The constants $C$ and $r_0$ are positive and depend only on $\omega$, $\omega_f$, $\omega_g$, the dimension, and ellipticity constants; in the second part, they also depend on the modulus of continuity of the coefficients of $\L$.
	\end{thm}

    \vspace{-0.3cm}
    
	\begin{obs}
		For example, if $f$ and $g$ are zero:
		\begin{itemize}
			\item In the limit case where the domain is only Lipschitz and $\omega(t) \leq L \ll 1$,
			$$\tilde\omega(t) = Cte^{CL\log t} = Ct^{1-CL},$$
			recovering in a quantitative form the boundary regularity results in \cite{Tor24}.
			\item If the modulus of continuity is Dini, the integral converges and we recover Lipschitz regularity \cite{Khi69,Naz12}.
			\item If $\omega(t) \leq C|\log t|^{-1}$, then $u$ is logLipschitz.
		\end{itemize}
	\end{obs}

    \vspace{-0.3cm}
	
	\subsection{Background}
	The issues of boundary nondegeneracy and boundary regularity for solutions to elliptic equations have been extensively studied, and there is a vast body of literature on the subject. We will provide an overview, and point the interested reader to the extensive and well-documented survey by Apushkinskaya and Nazarov \cite{AN22} (see also \cite{ABM+}). 

    \vspace{-0.13cm}
    
	\subsubsection{The boundary point lemma}
	The boundary point lemma, normal derivative lemma, or, more commonly, Hopf-Oleinik lemma is the following inequality:

    \vspace{-0.13cm}
    
	\begin{lem}
		Let $u$ be a positive harmonic function in $\Omega$, and assume that $u(x_0) = 0$ for a given $x_0 \in \p\Omega$. Then,
		$$\liminf\limits_{h\rightarrow0}\frac{u(x_0+h\nu)}{h} > 0,$$
		where $\nu$ denotes the inwards normal vector to $\p\Omega$ at $x_0$.
	\end{lem}

    \vspace{-0.13cm}
    
	The Hopf-Oleinik lemma can be understood as the boundary counterpart of the strong maximum principle, and its first version dates back to Zaremba in 1910 \cite{Zar10}, for harmonic functions in three dimensional domains with the interior ball condition. The breakthrough that gives name to the result came in 1952, when Hopf and Oleinik proved simultaneously and independently the normal derivative lemma for uniformly elliptic non-divergence form equations in domains with the interior ball condition \cite{Hop52,Ole52}.
	
	There have been many efforts to relax the interior ball condition, starting with the results by Keldysh and Lavrentiev in $C^{1,\alpha}$ domains. In \cite{Vyb62} and \cite{Vyb64}, Výborný showed the boundary point lemma under a technical assumption, which Lieberman later proved to be equivalent to $\p\Omega$ being $C^{1,\mathrm{Dini}}$ \cite{Lie85}. Widman (\cite{Wid67}), and Verzhbinskii and Maz'ya (\cite{VM67}) proved that the exterior $C^{1,\mathrm{Dini}}$ condition is necessary. Then, a series of generalizations came from Khimchenko \cite{Khi69,Khi70}, and later by Kamynin and Khimchenko (see \cite{KK80} and the references therein).
	
	For recent progress, we refer to a very general version of the lemma by Dong, Jeon and Vita \cite{DJV24}, where the $C^{1,\mathrm{Dini}}$ condition is relaxed to an averaged Dini condition that they call \textit{Dini Mean Oscillation}, and a counterexample showing the necessity of the exterior $C^{1,\mathrm{Dini}}$ condition in every direction by Apushkinskaya and Nazarov \cite{AN16}.
	
	Finally, let us mention a curious extension of the result to domains with the interior cone condition by Nadirashvili \cite{Nad83}: the Hopf-Oleinik lemma fails, but in every neighbourhood of $x_0$, there is $x_*$ where the inequality holds.
	
	\subsubsection{Boundary regularity}
	Lipschitz regularity at the boundary can be seen as the counterpart to the boundary point lemma. It was first established in domains with the exterior $C^{1,\mathrm{Dini}}$ condition by Khimchenko \cite{Khi69}. Furthermore, Lieberman proved that in $C^{1,\mathrm{Dini}}$ domains, solutions are $C^1$ up to the boundary \cite{Lie86,Lie87}. In this setting, Ma, Moreira, and Wang obtained explicit moduli of continuity for fully nonlinear elliptic and parabolic equations \cite{MW12,MMW17}.
	
	More recently, the minimal assumptions to ensure boundary differentiability (but not $C^1$) have been studied. Li and Wang started the endeavor proving that, in convex domains, solutions are always differentiable up to the boundary \cite{LW06,LW09}. This was extended to nonconvex domains by Li and Zhang \cite{LZ13} (with zero boundary condition), and then by Huang, Li, and Wang for general boundary data \cite{HLW14}. In the latter setting, they obtained a dichotomy: a harmonic function is differentiable at a boundary point if (i) the boundary blows up as a cone strictly contained in a half-space, or (ii) the boundary blows up as a hyperplane, and satisfies the exterior $C^{1,\mathrm{Dini}}$ condition.
	
	\subsubsection{Parallels with interior regularity}
	By scaling considerations, one can expect solutions to $\Delta u = f$ to be \textit{two derivatives more regular} than the right-hand side, but only \textit{as regular} as the boundary of the domain. This is apparent in the case of Schauder estimates: If $f \in C^{k,\alpha}$ and the boundary of the domain is $C^{k+2,\alpha}$, then $u \in C^{k+2,\alpha}$ \cite[Chapter 2]{FR22}.
	
	In this regard, the $C^1$ regularity up to the boundary in $C^{1,\mathrm{Dini}}$ domains matches the interior $C^2$ regularity of solutions to the Poisson equation with a Dini right-hand side \cite{Sha76,Bur78,Wan06}.
	
	In contrast, the parallel with interior regularity breaks at the level that we study: Solutions to the Poisson equation with a bounded right-hand side have their gradients in BMO, whereas harmonic functions in Lipschitz domains are only Hölder continuous up to the boundary. 
	
	\subsubsection{Strategies of proof}
	The first historical approach to boundary nondegeneracy and regularity estimates was the use of barriers adapted to the geometry of the domain, either directly, such as the truncated fundamental solution in an annulus for domains with the interior or exterior ball condition, or based on the regularized distance as in \cite{Lie85,Lie96}.
	
	A second approach inspired by Ladyzhenskaya and Uraltseva in their influential paper \cite{LU88}, was to approximate $\p\Omega$ by a hyperplane at dyadic scales, and keep track of the error of the approximations. In this construction, the Dini condition appears naturally as one needs to bound the sum of the errors,
	$$\sum\limits_{k \geq 0}\omega(2^{-k}) \lesssim \int_0^2\omega(t)\frac{\mathrm{d}t}{t},$$
	that converges when $\omega$ is a Dini modulus of continuity. Safonov combined this iterative approach with the boundary Harnack principle in \cite{Saf08}.
	
	Finally, the blow-up or contradiction-compactness-classification method is due to Leon Simon, \cite{Sim97}, and is based on exploiting the scaling of the equation. It has been very successful to prove estimates in Hölder spaces, also for nonlocal problems where the other methods are not suitable \cite{Ser15}.
	
	In this paper, we propose a new strategy that draws from both the first and the second tradition. We perform a dyadic decomposition, but rather than approximating the boundary by a hyperplane, we construct an adequate barrier at each scale based on a regularized distance. This allows us to derive estimates that are meaningful both in Lipschitz and in $C^1$ domains, bridging the gap between both frameworks.
	
	\subsection{Connection to free boundary problems}
	This work is partly motivated by the study of the singular set in the obstacle and thin obstacle problems, where one can prove that certain subsets of the free boundary are $C^1$ with a logarithmic modulus of continuity \cite{CSV18,CSV20}. We also expect logarithmic moduli of continuity for the normal vector in the Stefan problem \cite{HR19,FRT25}.
	
	In those settings, the regularity of solutions and of free boundaries are intimately related, and we hope that finer estimates for solutions will lead to better understanding of free boundaries as well. We plan to do that in a future work.
	
	\subsection{Acknowledgements.} This work has received funding from the European Research Council (ERC) under the Grant Agreement No 801867 and the Grant Agreement No 862342, from AEI project PID2021-125021NAI00 (Spain), from AEI project PID2024-156429NB-I00 (Spain), and from the Grant CEX2023-001347-S funded by MICIU/AEI/10.13039/501100011033 (Spain).
	
	We also wish to express our gratitude to Alberto Enciso, Xavier Ros-Oton and Tomás Sanz-Perela for their suggestions concerning this project.

    \newpage
	\subsection{Plan of the paper}
	This paper is organized as follows.
	
	We begin in Section \ref{sect:previ} by introducing our setting and recalling some tools that we will use, such as an almost positivity property and the boundary Harnack inequality. Then, in Section \ref{sect:cones}, we build barriers based on a localized regularized distance. Finally, Section \ref{sect:C1} is devoted to the proofs of our main results: Theorems \ref{thm:hopf} and~\ref{thm:upper}.
	
	\section{Preliminaries}\label{sect:previ}
	
	\subsection{Setting}
	Throughout the paper, the dimension will be $n \geq 2$. Given $x \in \R^n$, $B_r(x)$ will denote the ball of radius $r$ of $\R^n$, centered at $x$, and when $x = 0$ we will simply write $B_r$. Moreover, we will denote $x' = (x_1,\ldots,x_{n-1})$, and write $B'_r(x')$ and $B'_r$ for the balls of $\R^{n-1}$.
	
	Our starting point is the following definition of Lipschitz domains.
	\begin{defn}\label{defn:lip_domain}
		We say $\Omega \subset \R^n$ is a Lipschitz domain with Lipschitz constant $L$ if for any $x_0 \in \p\Omega$, there exists $r_0 > 0$ such that after a rotation and a translation that sends $x_0$ to the origin, there exists a Lipschitz function $\Gamma$ such that
		$$\Omega\cap B_{r_0} = \{x \in B_{r_0} : x_n > \Gamma(x')\} \quad \text{and} \quad \|\nabla\Gamma\|_{L^\infty} \leq L.$$
	\end{defn}
	
	For the sake of ease, we will only consider moduli of continuity as follows.
	\begin{defn}\label{defn:modulus}
		We say $\omega : [0,t_0) \to [0,\infty)$ is a modulus of continuity if it is continuous, strictly increasing, and $\omega(0) = 0$.
		
		We will write $f \in C^{k,\omega}$ if $D^kf$ is a continuous function such that $$\|D^kf(\cdot) - D^kf(x)\|_{L^\infty(B_r(x))} \leq \omega(r),$$
		for all $B_r(x)$ contained in the domain of $f$.
	\end{defn}
	
	We will consider domains satisfying the interior pointwise $C^1$ condition.
	\begin{defn}\label{defn:interiorC1}
		Let $\Omega$ be a Lipschitz domain. Then, we say $\Omega$ satisfies the interior $C^1$ condition at $x_0 \in \p\Omega$ with modulus of continuity $\omega$ if, after a rotation and a translation that sends $x_0$ to the origin,
		$$\{x_n > |x'|\omega(|x'|)\}\cap B_{t_0} \subset \Omega.$$
	\end{defn}
	
	Analogously, we can consider the exterior pointwise $C^1$ condition.
	\begin{defn}\label{defn:exteriorC1}
		Let $\Omega$ be a Lipschitz domain. Then, we say $\Omega$ satisfies the exterior $C^1$ condition at $x_0 \in \p\Omega$ with modulus of continuity $\omega$ if, after a rotation and a translation that sends $x_0$ to the origin,
		$$\{x_n < -|x'|\omega(|x'|)\}\cap B_{t_0}\cap\Omega = \emptyset.$$
	\end{defn}
    
    We will write $\p\Omega \in C^{1,\omega}$ to mean that $\Omega$ satisfies the interior and exterior $C^1$ conditions with modulus $\omega$ at every point in $\p\Omega$.
    
    \begin{obs}
        Separating the interior and exterior pointwise $C^1$ conditions helps us treat a wider class of domains than those whose boundary is locally a $C^1$ graph. For instance, if $\Omega$ is convex, it trivially satisfies the exterior $C^1$ condition at every point with $\omega \equiv 0$. On the contrary, both conditions fail at corner points of Lipschitz domains, where the normal vector has a jump discontinuity.
    \end{obs}
	
	In regularity theory, the following class of moduli of continuity deserves special attention:
	\begin{defn}\label{defn:Dini}
		Let $\omega$ be a modulus of continuity. We say $\omega$ is Dini if
		$$\int_0^{t_0}\omega(s)\frac{\mathrm{d}{s}}{s} < \infty.$$
	\end{defn}
	
	With this in hand, we can define a $C^{1,\mathrm{Dini}}$ domain.
	\begin{defn}\label{defn:dini_domain}
		We say $\Omega$ is a $C^{1,\mathrm{Dini}}$ domain if there exists a Dini modulus of continuity $\omega$ such that for every $x_0 \in \p\Omega$, $\Omega$ satisfies the interior and exterior $C^1$ conditions with modulus $\omega$.
	\end{defn}
	
	Finally, to work in the non-divergence setting, we introduce the extremal operators and the right notion of solutions.
	\begin{defn}\label{defn:Pucci}
		We will denote by
		$$\mathcal{M^-}(D^2u) := \inf\limits_{\lambda I \leq A \leq \Lambda I}\operatorname{Tr}(AD^2u), \quad \mathcal{M^+}(D^2u) := \sup\limits_{\lambda I \leq A \leq \Lambda I}\operatorname{Tr}(AD^2u)$$
		the Pucci extremal operators; see \cite{CC95} or \cite{FR22} for their properties.
	\end{defn}
	
	In our work, we will consider $L^n$-viscosity solutions.
	\begin{defn}[\cite{CCKS96}]\label{defn:Ln-viscosity}
		Let $u \in C(\Omega)$, $f \in L^n_{\mathrm{loc}}(\Omega)$, and $\mathcal{L}$ as in \eqref{eq:non-divergence_operator}. We say $u$ is a $L^n$-viscosity subsolution to $\L u = f$ (resp. supersolution) if, for all ${\varphi \in W^{2,n}_{\mathrm{loc}}(\Omega)}$ such that $u - \varphi$ has a local maximum (resp. minimum) at $x_0$,
		$$\operatorname{ess}\liminf\limits_{x \rightarrow x_0} \mathcal{L}\varphi - f \leq 0$$
		$$(\text{resp.} \ \operatorname{ess}\limsup\limits_{x \rightarrow x_0} \mathcal{L}\varphi - f \geq 0).$$
		
		We will say equivalently that $u$ is a solution to $\L u \leq (\geq) f$. When $u$ is both a subsolution and a supersolution, we say $u$ is a solution and write $\L u = f$.
	\end{defn}
	
	\subsection{Technical tools}
	We start with a simple computation.
	\begin{lem}\label{lem:exp_modulus_monotonicity}
		Let $\omega_1,\omega_2 : [0,t_0) \to [0,\infty)$ be moduli of continuity. Then, for every $a,c > 0$ and $b \in (0,t_0)$, there exists $\tilde t_0 > 0$ such that $\tilde\omega : [0,\tilde t_0) \to [0,\infty)$ defined by
		$$\tilde\omega(t) := t\left(a+\int_t^b\omega_1(s)\frac{\mathrm{d}s}{s}\right)\exp \left(c\int_t^b\omega_2(s)\frac{\mathrm{d}s}{s}\right)$$
		is also a modulus of continuity.
	\end{lem}
	
	\begin{proof}
		Since $\omega_1(0)=\omega_2(0)=0$, by continuity we can choose $\tilde t_0 \in (0,b)$ such that 
		$$\frac{\omega_1(\tilde t_0)}{a} + c\omega_2(\tilde t_0) \leq \frac{1}{2}.$$
		It is clear that $\tilde\omega$ is continuous. To see that $\tilde\omega$ is strictly increasing, we compute
		$$\left(\ln \tilde\omega(t)\right)' = \frac{1}{t} - \frac{\omega_1(t)}{t\left(a+\int_t^b\omega_1(s)\frac{\mathrm{d}s}{s}\right)} - \frac{c\omega_2(t)}{t} \geq \frac{1}{2t}.$$
		From here, it is also clear that $\lim\limits_{t\rightarrow0^+}\ln\tilde\omega(t) = -\infty$, and then $\tilde\omega(0) = 0$ as needed.
	\end{proof}
	
	We continue with the interior Harnack inequality, which can be found for \mbox{$L^n$-viscosity} solutions in \cite{Koi04}.
	\begin{thm}\label{thm:Harnack}
		Let $\L$ be as in \eqref{eq:non-divergence_operator}, and let $u$ be a nonnegative ($L^n$-viscosity) solution to $\L u = 0$ in $B_1$. Then, 
		$$\sup\limits_{B_{1/2}} u \leq C\inf\limits_{B_{1/2}} u,$$
		where the constant $C$ depends only on the dimension and ellipticity constants.
	\end{thm}
	
	Now we recall the Aleksandroff-Bakelman-Pucci estimate, see \cite[Theorem 3.2]{CC95} and \cite[Proposition 3.3]{CCKS96} for the full details and a proof.
	\begin{thm}\label{thm:ABP}
		Assume that $\Omega \subset \R^n$ is a bounded domain. Let $\L$ be as in \eqref{eq:non-divergence_operator}, and let $u \in C(\overline{\Omega})$ satisfy $\mathcal{L}u \geq f$ in the $L^n$-viscosity sense, with $f \in L^n(\Omega)$. Assume that $u$ is bounded on $\partial\Omega$.
		
		Then,
		$$\sup\limits_\Omega u \leq \sup\limits_{\p\Omega} u + C\operatorname{diam}(\Omega)\|f\|_{L^n(\Omega)}$$
		with $C$ only depending on the dimension and ellipticity constants.
	\end{thm}
	
	The following almost positivity property can be understood as a quantitative form of the maximum principle.
	
	\begin{lem}\label{lem:almost_positivity}
		Let $\L$ be as in \eqref{eq:non-divergence_operator}, and let $\Omega$ be a Lipschitz domain in the sense of Definition \ref{defn:lip_domain} with Lipschitz constant $L \leq \frac{1}{16}$ and $0 \in \p\Omega$. Let $u$ satisfy
		$$\left\{\begin{array}{rclll}
			\L u & = & 0 & \text{in} & \Omega\cap B_1\\
			u & \geq & \!-\mu_0 & \text{in} & \Omega\cap B_1\\
			u & = & 0 & \text{on} & \p\Omega\cap B_1\\
			u\left(\frac{e_n}{2}\right)\hspace{-0.3em} & \geq & 1.
		\end{array}\right.$$
		Then, 
		$$u \geq 0 \quad \text{in} \ \Omega\cap B_{1/2}.$$
		The constant $\mu_0 > 0$ depends only on the dimension and ellipticity constants.
	\end{lem}
	
	\begin{proof}
		Let $\delta \in (0,\frac{1}{3})$, to be chosen later. Then, by the interior Harnack inequality\footnote{See \cite[Lemma 3.6]{RT21} for an interior Harnack inequality whose constants do not depend on~$\Omega$.} applied to $v = u + \mu_0$, $v \geq \sqrt{\mu_0}+\mu_0$ in $\Omega_\delta := \{x \in B_{1-\delta} : \operatorname{dist}(x,\p\Omega) > \delta\}$. Thus, $\tilde u = \mu_0^{-1/2}u$ is $\L$-harmonic, vanishes on $\p\Omega$, and satisfies
		$$\left\{\begin{array}{rclll}
			\tilde u & \geq & 1 & \text{in} & \Omega_\delta\\
			\tilde u & \geq & -\sqrt{\mu_0} & \text{in} & \Omega.
		\end{array}\right.$$
		Then, choosing $\delta$ and $\mu_0$ adequately small, by \cite[Proposition B.8]{FR22}\footnote{In the original reference, the result is written for $\L = \Delta$, but the argument only uses the interior Harnack inequality, so it can be used for non-divergence equations with the same proof.} and a covering argument, $\tilde u \geq 0$ in $\Omega\cap B_{1/2}$, as we wanted to prove.
	\end{proof}
	
	Let us recall the boundary Harnack inequality. See for instance \cite{DS20} for a streamlined proof.
	\begin{thm}\label{thm:boundary_harnack}
		Let $\L$ be as in \eqref{eq:non-divergence_operator}, and let $\Omega$ be a Lipschitz domain in the sense of Definition \ref{defn:lip_domain} with Lipschitz constant $L \leq 1$ and $0 \in \p\Omega$. Let $u$ and $v$ be positive solutions to
		$$\left\{\begin{array}{rclll}
			\L u & = & 0 & \text{in} & \Omega\cap B_1\\
			u & = & 0 & \text{on} & \p\Omega\cap B_1
		\end{array}\right.\quad\text{and}\quad\left\{\begin{array}{rclll}
			\L v & = & 0 & \text{in} & \Omega\cap B_1\\
			v & = & 0 & \text{on} & \p\Omega\cap B_1.
		\end{array}\right.$$
		Then,
		$$\frac{1}{C}\frac{u\left(\frac{e_n}{2}\right)}{v\left(\frac{e_n}{2}\right)} \leq \frac{u}{v} \leq C\frac{u\left(\frac{e_n}{2}\right)}{v\left(\frac{e_n}{2}\right)} \quad \text{in} \ \Omega\cap B_{1/2},$$
		and
		$$\left\|\frac{u}{v}\right\|_{C^\alpha(B_{1/2})} \leq C\frac{u\left(\frac{e_n}{2}\right)}{v\left(\frac{e_n}{2}\right)},$$
		where $C$ and $\alpha$ are positive, and they depend only on the dimension and ellipticity constants.
	\end{thm}
	
	\section{Barriers}\label{sect:cones}
	
	The goal of this section is to construct sub- and supersolutions of the form $d^{1\pm\varepsilon}$, where $d$ is a regularized distance, and $\varepsilon$ depends quantitatively on the oscillation of the boundary of the domain.
	
	We first introduce a \textit{locally regularized distance}. The advantage of this construction with respect to the more usual technique by Lieberman is that the behaviour of $d$ depends only on the boundary at a scale comparable to $d$, leading to improved pointwise bounds when the regularity of $\p\Omega$ is not uniform.

	\begin{prop}\label{prop:regularized_distance}
		Let $\Omega$ be a Lipschitz domain in the sense of Definition \ref{defn:lip_domain} with Lipschitz constant $L \leq \frac{1}{C}$ and $0 \in \p\Omega$. Then, there exists a function $d : \Omega \cap B_{1/2} \to \R$ satisfying the following:
        $$1-C\|\nabla\Gamma\|_{L^\infty(B'_d(x'))} \leq \frac{d}{x_n - \Gamma(x')} \leq 1 + C\|\nabla\Gamma\|_{L^\infty(B'_d(x'))},$$
		\begin{align*}
			1-C\|\nabla\Gamma\|_{L^\infty(B'_d(x'))} &\leq |\nabla d| \leq 1 + C\|\nabla\Gamma\|_{L^\infty(B'_d(x'))},\\[0.05cm]
			|D^2d| &\leq C\frac{\|\nabla \Gamma\|_{L^\infty(B_d(x'))}}{d}.
		\end{align*}
		The constant $C$ depends only on the dimension.
	\end{prop}
	
	\begin{proof}
		We will construct $d$ as the inverse of a parametrization of $\Omega$.
		
		Let $\eta \in C^\infty_c(B_1)$ be such that $\eta \geq 0$ and $\int\eta = 1$. Now, let $P : B_1^+ \to \Omega$ be defined as $P(x',x_n) = (x',p(x))$, where
		$$p(x) = (\eta_{x_n}*\Gamma)(x')+x_n := \int_{\R^{n-1}}x_n^{1-n}\eta\left(\frac{z'}{x_n}\right)\Gamma(x'+z')\mathrm{d}z'+x_n.$$
		Then,
		$$0 = \p_n\int_{\R^{n-1}}\eta_{x_n} = \int_{\R^{n-1}}\frac{1-n}{x_n^n}\eta\left(\frac{z'}{x_n}\right)-\frac{z'}{x_n^{n+1}}\cdot\nabla\eta\left(\frac{z'}{x_n}\right),$$
		and thus
		\begin{align*}
			\p_np(x) &= \int_{\R^{n-1}}\left(\frac{1-n}{x_n^n}\eta\left(\frac{z'}{x_n}\right)-\frac{z'}{x_n^{n+1}}\cdot\nabla\eta\left(\frac{z'}{x_n}\right)\right)\Gamma(x'+z')\mathrm{d}z' + 1\\
			&= \int_{\R^{n-1}}\left(\frac{1-n}{x_n^n}\eta\left(\frac{z'}{x_n}\right)-\frac{z'}{x_n^{n+1}}\cdot\nabla\eta\left(\frac{z'}{x_n}\right)\right)\big(\Gamma(x'+z')-\Gamma(x')\big)\mathrm{d}z' + 1,
		\end{align*}
		which implies that
		$$1 - C\|\nabla\Gamma\|_{L^\infty(B'_{x_n}(x'))} \leq \p_np(x) \leq 1 + C\|\nabla\Gamma\|_{L^\infty(B'_{x_n}(x'))}.$$
		Hence, since $\operatorname{det}DP = \p_np > 0$, we can define $(y',d(y)) = P^{-1}(y',y_n)$.
		
		Now, to estimate the derivatives of $d$, we first estimate the derivatives of $p$:
		\begin{itemize}
			\item For $i = 1,\ldots,n-1$,
			$$|\p_ip| = \left|\int_{\R^{n-1}}x_n^{1-n}\eta\left(\frac{z'}{x_n}\right)\p_i\Gamma(x'+z')\mathrm{d}z'\right| \leq \|\nabla\Gamma\|_{L^\infty(B'_{x_n}(x'))}.$$
			\item For $i,j = 1,\ldots,n-1$, integrating by parts,
			$$|\p^2_{ij}p| = \left|\int_{\R^{n-1}}x_n^{-n}\p_i\eta\left(\frac{z'}{x_n}\right)\p_j\Gamma(x'+z')\mathrm{d}z'\right| \leq \frac{C\|\nabla\Gamma\|_{L^\infty(B'_{x_n}(x'))}}{x_n}.$$
			\item For $i = 1,\ldots,n-1$,
			\begin{align*}
				|\p^2_{in}p| &= \left|\int_{\R^{n-1}}\left(\frac{1-n}{x_n^n}\eta\left(\frac{z'}{x_n}\right)-\frac{z'}{x_n^{n+1}}\cdot\nabla\eta\left(\frac{z'}{x_n}\right)\right)\p_i\Gamma(x'+z')\mathrm{d}z'\right|\\[0.1cm]
				&\leq \frac{C\|\nabla\Gamma\|_{L^\infty(B'_{x_n}(x'))}}{x_n}.
			\end{align*}
			\item Finally,
			\begin{align*}
				|\p^2_{nn}p| &= \left|\int_{\R^{n-1}}\left(\frac{n(n-1)}{x_n^{n+1}}\eta\left(\frac{z'}{x_n}\right)+\frac{2nz'}{x_n^{n+2}}\cdot\nabla\eta\left(\frac{z'}{x_n}\right)\right.\right.\\[0.1cm]
				&\left.\left.\qquad +\frac{1}{x_n^{n+3}}(z')^\top\cdot D^2\eta\left(\frac{z}{x_n}\right)\cdot z'\right)\Gamma(x'+z')\mathrm{d}z'\right|\\[0.1cm]
				&\leq \frac{C\|\nabla\Gamma\|_{L^\infty(B'_{x_n}(x'))}}{x_n}.
			\end{align*}
		\end{itemize}
		
		The first derivatives of $d$ can be estimated as:
		\begin{itemize}
			\item For the last component,
			$$|\p_nd| = \frac{1}{|\p_np|} \leq 1+C\|\nabla\Gamma\|_{L^\infty(B'_{x_n}(x'))}.$$
			\item For $i = 1,\ldots,n-1$, 
			$$|\p_id| = \left|\frac{\p_ip}{\p_np}\right| \leq C\|\nabla\Gamma\|_{L^\infty(B'_{x_n}(x'))}.$$
		\end{itemize}
		
		And for $D^2d$ we distinguish three cases:
		\begin{itemize}
			\item When $i, j = 1,\ldots,n-1$,
			$$|\p^2_{ij}d| \leq \frac{1}{|\p_np|}\left(|\p^2_{ij}p|+\frac{|\p^2_{in}p\hspace{0.1em}\p_jp|+|\p^2_{nj}p\hspace{0.1em}\p_ip|}{|\p_np|}+\frac{|\p^2_{nn}p\hspace{0.1em}\p_ip\hspace{0.1em}\p_jp|}{|\p_np|^2}\right) \leq \frac{C\|\nabla\Gamma\|_{L^\infty(B'_{x_n}(x'))}}{x_n}.$$
			\item When $i = 1,\ldots,n-1$,
			$$|\p^2_{in}d| \leq \frac{1}{|\p_np|^2}\left(|\p_{in}p|+\frac{|\p^2_{nn}p\hspace{0.1em}\p_ip|}{|\p_np|}\right) \leq \frac{C\|\nabla\Gamma\|_{L^\infty(B'_{x_n}(x'))}}{x_n}.$$
			\item Finally,
			$$|p^2_{nn}d| = \frac{|\p^2_{nn}p|}{|\p_np|^3} \leq \frac{C\|\nabla\Gamma\|_{L^\infty(B'_{x_n}(x'))}}{x_n}.$$
		\end{itemize}
	\end{proof}

	Using the previous estimates, we see that $d^{1+\varepsilon}$ is a subsolution and $d^{1-\varepsilon}$, a supersolution.
	\begin{lem}\label{lem:barriers}
		For every Lipschitz domain $\Omega$ with Lipschitz constant $L \leq \frac{1}{C_0}$, and for any $\varepsilon \geq C_0\|\nabla\Gamma\|_{L^\infty(B'_{2r})},$
		$$\mathcal M^-d^{1+\varepsilon} \geq 0 \quad \text{and} \quad \mathcal M^+d^{1-\varepsilon} \leq 0 \quad \text{in} \ B_r,$$
        \vspace{-0.05cm}
		where the constant $C_0$ depends only on the dimension and ellipticity constants.
	\end{lem}
	
	\begin{proof}
		Using the definition of the Pucci operator and Proposition \ref{prop:regularized_distance},
		\begin{align*}
			\mathcal M^-d^{1+\varepsilon} &= \inf\limits_{\lambda I \leq A \leq \Lambda I}\operatorname{Tr}(AD^2d^{1+\varepsilon})\\[-0.1cm]
			&= (1+\varepsilon)\inf\limits_{\lambda I \leq A \leq \Lambda I}\left[\sum\limits_{i=1}^n\sum\limits_{j=1}^na_{ij}(d^\varepsilon\p^2_{ij}d + \varepsilon d^{-1+\varepsilon}\p_id\p_jd)\right]\\[-0.05cm]
			&\geq (1+\varepsilon)\left(d^\varepsilon\mathcal M^-d+\varepsilon d^{-1+\varepsilon}\inf\limits_{\lambda I \leq A \leq \Lambda I}\nabla d^\top A \nabla d\right)\\
			&\geq (1+\varepsilon)\left(-C\Lambda \|\nabla\Gamma\|_{L^\infty(B_d(x'))}+\frac{\lambda\varepsilon}{4}\right)d^{-1+\varepsilon} \geq 0.
		\end{align*}
		The inequality for $\mathcal M^+d^{1-\varepsilon}$ is analogous.
	\end{proof}
	
	Then, using these barriers we construct a solution with a controlled growth.
	\begin{prop}\label{prop:special_solns}
		Let $\L$ be as in \eqref{eq:non-divergence_operator}, and let $\Omega$ be a Lipschitz domain in the sense of Definition \ref{defn:lip_domain} with Lipschitz constant $L \leq L_0$. Then, for every $r \in (0,\frac{1}{3})$, there exists a solution $\varphi_r$ to
        \vspace{-0.05cm}
		$$\left\{\begin{array}{rclll}
			\L\varphi_r & = & 0 & \text{in} & \Omega\cap B_r\\
			\varphi_r & = & 0 & \text{on} & \p\Omega\cap B_r
		\end{array}\right.$$
		with the growth estimates
		$$(2r)^{-\varepsilon}d^{1+\varepsilon} \leq \varphi_r \leq (2r)^\varepsilon d^{1-\varepsilon} \quad \text{in} \ \Omega\cap B_r,$$
		and
		$$\|\varphi_r - d\|_{L^\infty(\Omega\cap B_r)} \leq Kr\|\nabla\Gamma\|_{L^\infty(B'_{2r})},$$
		where $\varepsilon = C_0\|\nabla\Gamma\|_{L^\infty(B'_{2r})}$, and $C_0$ is from Lemma \ref{lem:barriers}. The constants $L_0$ and $K$ depend only on the dimension and ellipticity constants.
	\end{prop}
	
	\begin{proof}
		First, by Proposition \ref{prop:regularized_distance}, $d \leq 2r$ in $\Omega\cap B_r$, and then $(2r)^{-\varepsilon}d^{1+\varepsilon} \leq (2r)^\varepsilon d^{1-\varepsilon}$. Now, let $\varphi_r$ be the solution to
		$$\left\{\begin{array}{rclll}
			\L\varphi_r & = & 0 & \text{in} & \Omega\cap B_r\\
			\varphi_r & = & d & \text{on} & \p(\Omega\cap B_r).
		\end{array}\right.$$
		From Lemma \ref{lem:barriers} and the comparison principle, it follows that
		$$(2r)^{-\varepsilon}d^{1+\varepsilon} \leq \varphi_r \leq (2r)^\varepsilon d^{1-\varepsilon} \quad \text{in} \ \Omega\cap B_r.$$
		Moreover,
        \vspace{-0.1cm}
		\begin{align*}
			\|\varphi_r - d\|_{L^\infty(\Omega\cap B_r)} &\leq \|(2r)^{-\varepsilon}d^{1+\varepsilon}-(2r)^\varepsilon d^{1-\varepsilon}\|_{L^\infty(\Omega\cap B_r)}\\[0.05cm]
			&\leq \sup\limits_{t \in [0,2r]}(2r)^{-\varepsilon}t^{1+\varepsilon}-(2r)^\varepsilon t^{1-\varepsilon} \leq Cr\varepsilon.\qedhere
		\end{align*}
	\end{proof}

    
	\section{Quantitative nondegeneracy and regularity in \texorpdfstring{$C^1$}{C1} domains}\label{sect:C1}
	
	We first prove our boundary nondegeneracy estimate for domains with the interior $C^1$ condition.
	
	\begin{proof}[Proof of Theorem \ref{thm:hopf}]
		Let $r_0 > 0$ small enough such that $\omega(r_0) < \omega_0$ to be determined later. Without loss of generality after rescaling and dividing by a constant, we may assume that $u(\frac{e_n}{4}) = 1$ and $r_0 = 1$. Then, we define the sequences $\varphi_k := \varphi_{2^{-k-1}}, \varepsilon_k := C_0\|\nabla\Gamma\|_{L^\infty(B'_{2^{-k}})}$ as introduced in Proposition \ref{prop:special_solns}.
		
		First, by the boundary Harnack inequality (Theorem \ref{thm:boundary_harnack}),
		$$u \geq c_0\varphi_0\chi_{B_{1/4}}.$$
		
		Now, we will see inductively that for $k \geq 1$,
		$$u \geq c_k\varphi_k\chi_{B_{2^{-k-2}}}, \quad c_k := (1-A\varepsilon_{k-1})c_{k-1},$$
		where $A$ is a large constant to be chosen later depending only on the dimension and ellipticity constants.
		
		Indeed, let
		$$v(x) := \frac{\varphi_{k-1}(2^{-k-1}x)}{\varphi_{k-1}(2^{-k-2}e_n)}, \quad \text{and} \quad w(x) := \frac{\varphi_k(2^{-k-1}x)}{\varphi_k(2^{-k-2}e_n)}.$$
		By construction, $v(\frac{e_n}{2}) = w(\frac{e_n}{2}) = 1$, and by Proposition \ref{prop:special_solns},
		\begin{align*}
			\|v - w\|_{L^\infty(\tilde\Omega\cap B_1)} &= \left\|\frac{\varphi_{k-1}}{\varphi_{k-1}(2^{-k-2}e_n)} - \frac{\varphi_k}{\varphi_k(2^{-k-2}e_n)}\right\|_{L^\infty(\Omega\cap B_{2^{-k-1}})}\\
			&\leq \left|\frac{1}{\varphi_{k-1}(2^{-k-2}e_n)} - \frac{1}{\varphi_k(2^{-k-2}e_n)}\right|\max\{\|\varphi_{k-1}\|_{L^\infty},\|\varphi_k\|_{L^\infty}\}\, +\\
			&\quad+\max\left\{\frac{1}{\varphi_{k-1}(2^{-k-2}e_n)}, \frac{1}{\varphi_k(2^{-k-2}e_n)}\right\}\|\varphi_{k-1}-\varphi_k\|_{L^\infty}\\
			&\leq \frac{\|\varphi_{k-1}-\varphi_k\|_{L^\infty}}{\varphi_{k-1}(2^{-k-2}e_n)\varphi_k(2^{-k-2}e_n)}\max\{\|\varphi_{k-1}\|_{L^\infty},\|\varphi_k\|_{L^\infty}\}\, +\\
			&\quad+\max\left\{\frac{1}{\varphi_{k-1}(2^{-k-2}e_n)}, \frac{1}{\varphi_k(2^{-k-2}e_n)}\right\}\|\varphi_{k-1}-\varphi_k\|_{L^\infty}\\
			&\leq \frac{2\cdot2^{-k}K\|\nabla\Gamma\|_{L^\infty(B_{2^{-k+1}})}}{(2^{-k-3}-2^{-k}K\|\nabla\Gamma\|_{L^\infty(B_{2^{-k+1}})})^2}2^{-k}\, +\\
			&\quad+\frac{1}{2^{-k-3}-2^{-k}K\|\nabla\Gamma\|_{L^\infty(2^{-k+1})}}2\cdot2^{-k}K\|\nabla\Gamma\|_{L^\infty(B_{2^{-k+1}})}\\
			&\leq M\varepsilon_{k-1},
		\end{align*}
		where we wrote $L^\infty$ instead of $L^\infty(B_{2^{-k-1}})$ for formatting.
		
		Now, let
		$$h := \frac{v - (1-\mu_0^{-1}M\varepsilon_{k-1})w}{\mu_0^{-1}M\varepsilon_{k-1}}.$$
		Denoting by $\tilde\L$ and $\tilde\Omega$ the appropriate rescalings of $\L$ and $\Omega$, $h$ is a solution to
		$$\left\{\begin{array}{rclll}
			\tilde\L h & = & 0 & \text{in} & \tilde\Omega\cap B_1\\
			h & \geq & \!-\mu_0 & \text{in} & \tilde\Omega\cap B_1\\
			h & = & 0 & \text{on} & \p\tilde\Omega\cap B_1\\
			h(\frac{e_n}{2})\hspace{-0.3em} & = & 1, &&
		\end{array}\right.$$
		and hence, by Lemma \ref{lem:almost_positivity}, $h \geq 0$ in $\tilde\Omega\cap B_{1/2}$, in other words, $v \geq (1-\mu_0^{-1}M\varepsilon_{k-1})w$ in $\tilde\Omega\cap B_{1/2}$.
		
		Going back to our inductive argument, if we assume that $u \geq c_{k-1}\varphi_{k-1}$ in ${\Omega\cap B_{2^{-k-1}}}$, this implies that
		\begin{align*}
			u &\geq c_{k-1}\varphi_{k-1}(2^{-k-2}e_n)v(2^{k+1}\cdot) \geq c_{k-1}\varphi_{k-1}(2^{-k-2}e_n)(1-\mu_0^{-1}M\varepsilon_{k-1})w(2^{k+1}\cdot)\\
			&\geq c_{k-1}\frac{\varphi_{k-1}(2^{-k-2}e_n)}{\varphi_k(2^{-k-2}e_n)}(1-\mu_0^{-1}M\varepsilon_{k-1})\varphi_k\\
			&\geq c_{k-1}\frac{2^{-k-3}-2^{-k}K\|\nabla\Gamma\|_{L^\infty(2^{-k+1})}}{2^{-k-3}+2^{-k-1}K\|\nabla\Gamma\|_{L^\infty(2^{-k})}}(1-\mu_0^{-1}M\varepsilon_{k-1})\varphi_k \geq c_{k-1}(1 - A\varepsilon_{k-1})\varphi_k,
		\end{align*}
		in $\Omega\cap B_{2^{-k-2}}$, provided that $\varepsilon_{k-1}$ is small enough.
		
		Then choosing $\omega_0$ small enough (so that $\varepsilon_{j-1} < 1/(2A)$ for all $j$), we compute
		$$c_k = c_0\prod\limits_{j = 0}^{k-1}(1 - A\varepsilon_j) \geq c_0\cdot 4^{-A\sum\limits_{j=0}^{k-1}\varepsilon_j}.$$
		Thus, for any $r \in [2^{-k-3},2^{-k-2}]$,\vspace{-0.3cm}
		\begin{align*}
			u(re_n) &\geq c_k\varphi_k(re_n) \geq c_0\cdot 4^{-A\sum\limits_{j = 0}^{k-1}\varepsilon_j}\varphi_k(re_n)\\
			&\geq c_0\cdot 4^{-A\sum\limits_{j = 0}^{k-1}\varepsilon_j}(2^{-k})^{-\varepsilon_k}\left(\frac{r}{2}\right)^{1+\varepsilon_k} \geq \frac{c_0}{32}\cdot4^{-A\sum\limits_{j = 0}^{k-1}\varepsilon_j}r,
		\end{align*}
		and by Proposition \ref{prop:special_solns},
		$$\frac{u(re_n)}{r} \geq \frac{1}{C}\exp -C\sum\limits_{j=0}^{k-1}\omega(2^{-j})\geq \frac{1}{C}\exp \left(-C\int_{8r}^2\omega(s)\frac{\mathrm{d}s}{s}\right).$$
	\end{proof}
	
	The proof of the upper bound for domains with the exterior $C^1$ condition follows the same strategy, with some extra work to deal with the Dirichlet boundary condition and the source term.

    \begin{proof}[Proof of Theorem \ref{thm:upper}]    
        Let $r_0 > 0$ small enough such that $\omega(r_0), \omega_g(r_0), \omega_f(r_0) < \omega_0$ to be determined later. Without loss of generality, after rescaling and subtracting a linear function, we may assume that $r_0 = 1$, $g(0) = 0$ and $\nabla g(0) = 0$. Then, we define the sequences $\varphi_k := \varphi_{2^{-2k-1}}, \varepsilon_k := C_0\|\nabla\Gamma\|_{L^\infty(B'_{4^{-k}})}$ as introduced in Proposition~\ref{prop:special_solns}.

        First, let $\tilde u_0$ be the solution to
        $$\left\{\begin{array}{rclll}
			\L \tilde u_0 & = & 0 & \text{in} & \Omega \cap B_{1/2}\\
			\tilde u_0 & = & \!|u| & \text{on} & \p B_{1/2} \cap \Omega\\
			\tilde u_0 & = & 0 & \text{on} & \p\Omega \cap B_{1/2}
		\end{array}\right.$$
        
		Then, by the boundary Harnack inequality (Theorem \ref{thm:boundary_harnack}),
		$$\tilde u_0 \leq c_0\varphi_0 \quad \text{in} \ \Omega\cap B_{1/4},$$
		where $c_0 \leq C\|u\|_{L^\infty(B_{1/2})}$, and by the comparison principle and the ABP estimate (Theorem \ref{thm:ABP}),
		$$|u| \leq \tilde u_0 + \|g\|_{L^\infty(B_{1/2})} + C\|f\|_{L^n(\Omega\cap B_{1/2})} \leq c_0\varphi_0 + d_0 \quad \text{in} \ \Omega\cap B_{1/4},$$
		where $d_0 = \omega_g(1)+C\omega_f(1)$.

        Now, we will see inductively that for $k \geq 1$,
		$$|u| \leq c_k\varphi_k + 4^{-k}d_k \quad \text{in} \ \Omega\cap B_{4^{-k-1}},$$
		where 
		$$c_k := (1+A\varepsilon_{k-1})c_{k-1} + Ad_{k-1} \quad \text{and} \quad d_k := \omega_g(4^{-k})+C\omega_f(4^{-k}),$$
		and $A$ is a large constant to be chosen later depending only on the dimension and ellipticity constants.

        Indeed, assume by induction that $|u| \leq c_{k-1}\varphi_{k-1} + 4^{-k+1}d_{k-1}$ in $\Omega\cap B_{4^{-k}}$. Then, let $\tilde u_k$ be the solution to
        $$\left\{\begin{array}{rclll}
			\L \tilde u_k & = & 0 & \text{in} & \Omega \cap B_{4^{-k}}\\
			\tilde u_k & = & \!|u| & \text{on} & \p B_{4^{-k}} \cap \Omega\\
			\tilde u_k & = & 0 & \text{on} & \p\Omega \cap B_{4^{-k}},
		\end{array}\right.$$
        and note that also $\tilde u_k \leq c_{k-1}\varphi_{k-1} + 4^{-k+1}d_{k-1}$ in $\Omega\cap B_{4^{-k}}$.

        Now, since $\varphi_{k-1}(2^{-2k-1}e_n) \simeq 2^{-2k-1}$, by the boundary Harnack inequality applied to $\tilde u_k - c_{k-1}\varphi_{k-1}$ and $\varphi_{k-1}$, 
        $$\tilde u_k - c_{k-1}\varphi_{k-1} \lesssim \frac{4^{-k+1}d_{k-1}}{\varphi_{k-1}(2^{-2k-1}e_n)}\varphi_{k-1} \lesssim d_{k-1}\varphi_{k-1} \quad \text{in} \ \Omega\cap B_{2^{-2k-1}},$$
        in other words, $\tilde u_k \leq (c_{k-1} + Cd_{k-1})\varphi_{k-1}$ in $\Omega\cap B_{2^{-2k-1}}$.

        Then, we use the same strategy as in the proof of Theorem \ref{thm:hopf}. Let
        $$v := \frac{\varphi_{k-1}(2^{-2k-1}x)}{\varphi_{k-1}(2^{-2k-2}e_n)}, \quad \text{and} \quad w(x) := \frac{\varphi_k(2^{-2k-1}x)}{\varphi_k(2^{-2k-2}e_n)}.$$
        By an analogous computation, $v(\frac{e_n}{2}) = w(\frac{e_n}{2}) = 1$, and $\|v - w\|_{L^\infty(\tilde\Omega\cap B_1)} \leq M\varepsilon_{k-1}$.

        Now, let
		$$h := \frac{(1+\mu_0^{-1}M\varepsilon_{k-1})w - v}{\mu_0^{-1}M\varepsilon_{k-1}}.$$
		Denoting by $\tilde\L$ and $\tilde\Omega$ the appropriate rescalings of $\L$ and $\Omega$, $h$ is a solution to
		$$\left\{\begin{array}{rclll}
			\tilde\L h & = & 0 & \text{in} & \tilde\Omega\cap B_1\\
			h & \geq & \!-\mu_0 & \text{in} & \tilde\Omega\cap B_1\\
			h & = & 0 & \text{on} & \p\tilde\Omega\cap B_1\\
			h(\frac{e_n}{2})\hspace{-0.3em} & = & 1, &&
		\end{array}\right.$$
		and hence, by Lemma \ref{lem:almost_positivity}, $h \geq 0$ in $\tilde\Omega\cap B_{1/2}$, in other words, $v \leq (1+\mu_0^{-1}M\varepsilon_{k-1})w$ in $\tilde\Omega\cap B_{1/2}$.

        Now we undo the scaling and compute in the same way as in the proof of Theorem~\ref{thm:hopf}:
        \begin{align*}
            \tilde u_k &\leq (c_{k-1} + Cd_{k-1})\varphi_{k-1} \leq (c_{k-1} + Cd_{k-1})(1 + \mu_0^{-1}M\varepsilon_{k-1})\frac{\varphi_{k-1}(2^{-2k-2}e_n)}{\varphi_k(2^{-2k-2}e_n)}\varphi_k\\
            &\leq \left[(1+A\varepsilon_{k-1})c_{k-1} + Ad_{k-1}\right]\varphi_k,
        \end{align*}
        in $\Omega \cap B_{4^{-k-1}}$, provided that $\varepsilon_{k-1}$ is small enough. Hence, by the comparison principle and the ABP estimate,
		$$|u| \leq c_k\varphi_k + 4^{-k}d_k \quad \text{in} \ \Omega\cap B_{4^{-k-1}}.$$

        Then we compute
		$$c_k = c_0\prod\limits_{j=0}^{k-1}(1+A\varepsilon_j) + A\sum\limits_{i=0}^{k-1}d_i\prod\limits_{j=i+1}^{k-1}(1+A\varepsilon_j) \leq c_0 e^{A\sum\limits_{j=0}^{k-1}\varepsilon_j} + A\sum\limits_{i=0}^{k-1}d_ie^{A\sum\limits_{j={i+1}}^{k-1}\varepsilon_j}.$$
		Thus, for any $r \in [4^{-k-2},4^{-k-1}]$,
		\begin{align*}
			\|u\|_{L^\infty(B_r)} &\leq c_k\|\varphi_k\|_{L^\infty(B_r)} + 4^{-k}d_k\\
			&\leq \left(c_0 e^{A\sum\limits_{j=0}^{k-1}\varepsilon_j} + A\sum\limits_{i=0}^{k-1}d_ie^{A\sum\limits_{j={i+1}}^{k-1}\varepsilon_j}\right)\|\varphi_k\|_{L^\infty(B_r)} + 4^{-k}d_k\\
			&\leq \left(c_0 e^{A\sum\limits_{j=0}^{k-1}\varepsilon_j} + A\sum\limits_{i=0}^{k-1}d_ie^{A\sum\limits_{j={i+1}}^{k-1}\varepsilon_j}\right)(4^{-k})^{\varepsilon_k}(2r)^{1-\varepsilon_k} + 4^{-k}d_k\\
			&\leq 16\left(c_0 e^{A\sum\limits_{j=0}^{k-1}\varepsilon_j} + A\sum\limits_{i=0}^{k-1}d_ie^{A\sum\limits_{j={i+1}}^{k-1}\varepsilon_j} + d_k\right)r,
		\end{align*}
        and by Proposition \ref{prop:special_solns},		
		\begin{align*}
			\frac{\|u\|_{L^\infty(B_r)}}{r} &\leq C\left(\|u\|_{L^\infty(B_{1/2})}\exp \left(C\int_{16r}^{2}\omega(s)\frac{\mathrm{d}s}{s}\right) + \int_{4r}^{2}[\omega_g+\omega_f](s)e^{C\left(\int_{16r}^{s/2}\omega(t)\frac{\mathrm{d}t}{t}\right)_+}\frac{\mathrm{d}s}{s}\right)\\
			&\leq C\left(\|u\|_{L^\infty(B_1)} + \int_r^2[\omega_g+\omega_f](s)\frac{\mathrm{d}s}{s}\right)\exp \left(C\int_r^2\omega(s)\frac{\mathrm{d}s}{s}\right).
		\end{align*}
        
		Finally, let $x,y \in \overline{\Omega}\cap B_{1/2}$, such that $d = |x-y| < \frac{r_0}{8}$. Assume without loss of generality that $d_x := \operatorname{dist}(x,\p\Omega) \geq \operatorname{dist}(y,\p\Omega)$. We distinguish three cases:
        \begin{itemize}
            \item If $d_x \geq \frac{r_0}{4}$, we write $u = \bar u + w$, where
            $$\left\{\begin{array}{rclll}
            \L \bar u & = & 0 & \text{in} & B_{r_0/4}(x)\\
            \bar u & = & u & \text{on} & \p B_{r_0/4}(x)
            \end{array}\right.
            \quad\text{and}\quad
            \left\{\begin{array}{rclll}
            \L w & = & f & \text{in} & B_{r_0/4}(x)\\
            w & = & 0 & \text{on} & \p B_{r_0/4}(x).
            \end{array}\right.$$
        Now, since the coefficients of $\L$ are Dini continuous, we can apply \cite[Theorem 1.1]{DK21} (in $\R^2$) or \cite[Theorem 1.9]{HK20} (in higher dimensions) to write
        $$w(y) = \int_{B_{r_0/4}(x)}G(y,z)f(z)\mathrm{d}z,$$
        with $|G(y,z)| \lesssim 1 + \left|\log|y-z|\right|$ in $\R^2$, $|G(y,z)| \lesssim |y-z|^{2-n}$ in higher dimensions, and $|\nabla_y G(y,z)| \lesssim |y-z|^{1-n}$ in all cases.

        In two dimensions,
        \begin{align*}
            |w(x) - w(y)| &\leq \int_{B_{r_0/4}(x)}|G(x,z) - G(y,z)||f(z)|\mathrm{d}z\\
            &\leq \int_{B_{3d}(x)}|G(x,z) - G(y,z)||f(z)|\mathrm{d}z\\
            &\qquad+ \int_{B_{r_0/4}(x) \cap \{|(x+y)/2-z| > 2d\}}|G(x,z) - G(y,z)||f(z)|\mathrm{d}z\\
            &\lesssim \int_{B_{3d}(x)}(1 + \left|\log|x-z|\right| + \left|\log|y-z|\right|)|f(z)|\mathrm{d}z\\
            &\qquad+ \int_{B_{r_0/4}(x)\cap\{|(x+y)/2-z| > 2d\}}\frac{d}{|x-z|}|f(z)|\mathrm{d}z\\
            &\lesssim \left(\int_{B_{3d}}(1+|\log|z||)^2\mathrm{d}z\right)^{1/2}\omega_f(3d)\\
            &\quad+ \sum\limits_{k = 1}^\infty \left(\int_{B_{r_0/4}\cap\{2^kd < |(x+y)/2-z| < 2^{k+1}d\}}2^{-2k}\mathrm{d}z\right)^{1/2}\omega_f(2^{k+1}d)\\
            &\lesssim d\omega_f(3d) + d\sum\limits_{k = 1}^M\omega_f(2^{k+1}d) \lesssim d\int_d^{r_0}\omega_f(s)\frac{\mathrm{d}s}{s},
        \end{align*}
        where $M$ is the largest integer such that $2^Md \leq \frac{5}{16}r_0$.\footnote{Note that $\operatorname{dist}\left(\frac{x+y}{2},\p B_{r_0/4}(x)\right) \leq \frac{r_0}{4} + \frac{d}{2} < \frac{5}{16}r_0$.} In higher dimensions, a similar strategy gives
        \begin{align*}
            |w(x) - w(y)| &\lesssim \int_{B_{3d}(x)}\left(\frac{1}{|x-z|^{n-2}}+\frac{1}{|y-z|^{n-2}}\right)|f(z)|\mathrm{d}z\\
            &\qquad+ \int_{B_{r_0/4}(x)\cap\{|(x+y)/2-z| > 2d\}}\frac{d}{|x-z|^{n-1}}|f(z)|\mathrm{d}z\\
            &\lesssim \left(\int_{B_{3d}}\frac{1}{|z|^{n(n-2)/(n-1)}}\mathrm{d}z\right)^{(n-1)/n}\omega_f(3d)\\
            &\quad+ d\sum\limits_{k = 1}^\infty \left(\int_{B_{r_0/4}\cap\{2^kd < |(x+y)/2-z| < 2^{k+1}d\}}2^{-nk}d^{-n}\right)^{(n-1)/n}\omega_f(2^{k+1}d)\\
            &\lesssim d\omega_f(3d) + d\sum\limits_{k = 1}^M\omega_f(2^{k+1}d) \lesssim d\int_d^{r_0}\omega_f(s)\frac{\mathrm{d}s}{s}.
        \end{align*}
            
        Then, we use interior Lipschitz estimates to compute
        \begin{align*}
            |u(x) - u(y)| &\leq |\bar u(x) - \bar u(y)| + |w(x) - w(y)|\\
            &\leq d\|\nabla \bar u\|_{L^\infty(B_{r_0/4}(x))} + |w(x) - w(y)|\\
            &\lesssim d\left(\|\bar u\|_{L^\infty(B_{r_0/2}(x))} + \int_d^{r_0}\omega_f(s)\frac{\mathrm{d}s}{s}\right)\\
            &\lesssim d\left(\|u\|_{L^\infty(B_1)} + \int_d^{r_0}\omega_f(s)\frac{\mathrm{d}s}{s}\right).
        \end{align*}

        \item If $d_x < \frac{r_0}{4}$ and $d \geq \frac{d_x}{2}$, we first bound
			$$|v(z)| \leq Cd_x\left(\|u\|_{L^\infty(B_1)}+\int_{d_x}^{2r_0}[\omega_g+\omega_f](s)\frac{\mathrm{d}s}{s}\right)\exp \left(C\int_{d_x}^{2r_0}\omega(s)\frac{\mathrm{d}s}{s}\right),$$
			for $z = x, y$. Therefore,
			\begin{align*}
				|u(x) - u(y)| &\leq |v(x)| + |v(y)| + |\nabla g(0)\cdot(x-y)|\\
				&\lesssim d_x\left(\|u\|_{L^\infty(B_1)}+\int_{d_x}^{2r_0}[\omega_g+\omega_f](s)\frac{\mathrm{d}s}{s}\right)\exp \left(C\int_{d_x}^{2r_0}\omega(s)\frac{\mathrm{d}s}{s}\right)\\
                &\qquad + d|\nabla g(0)|.
			\end{align*}
                Note that 
                $$|\nabla g(0)| \leq r_0^{-1}\|g - g(0)\|_{L^\infty(B_{r_0})} + \omega_g(r_0) \leq r_0^{-1}\|u\|_{L^\infty(B_1)} + \frac{1}{\ln 2}\int_{r_0}^{2r_0}\omega_g(s)\frac{\mathrm{d}s}{s}.$$
			Then, by Lemma \ref{lem:exp_modulus_monotonicity} if $d > d_x$, and trivially otherwise, it follows that
			$$|u(x) - u(y)| \lesssim d\left(\|u\|_{L^\infty(B_1)} + \int_d^{2r_0}[\omega_g+\omega_f](s)\frac{\mathrm{d}s}{s}\right)\exp \left(C\int_d^{2r_0}\omega(s)\frac{\mathrm{d}s}{s}\right).$$

            \item If $d < \frac{d_x}{2} < \frac{r_0}{8}$, let $x_* \in \p\Omega$ such that $|x-x_*| = d_x$. Then, we write $u = \bar u + w$, where
            $$\left\{\begin{array}{rclll}
            \L \bar u & = & 0 & \text{in} & B_{d_x(x)}\\
            \bar u & = & u & \text{on} & \p B_{d_x(x)}
            \end{array}\right.
            \quad\text{and}\quad
            \left\{\begin{array}{rclll}
            \L w & = & f & \text{in} & B_{d_x(x)}\\
            w & = & 0 & \text{on} & \p B_{d_x(x)}.
            \end{array}\right.$$
            Then, as in the first case,
            $$|w(x) - w(y)| \leq C d\int_d^{4d_x}\omega_f(s)\frac{\mathrm{d}s}{s}.$$
            On the other hand,
            \begin{align*}
				|\bar u(x)-\bar u(y)| &\leq d\|\nabla \bar u\|_{L^\infty(B_{d_x/2}(x))} \lesssim \frac{d}{d_x}\|\bar u\|_{L^\infty(B_{d_x}(x))} \leq \frac{d}{d_x}\|u\|_{L^\infty(B_{2d_x}(x_*))}\\
				&\lesssim d\left(\|u\|_{L^\infty(B_1)}+\int_{2d_x}^{2r_0}[\omega_g+\omega_f](s)\frac{\mathrm{d}s}{s}\right)\exp \left(C\int_{2d_x}^{2r_0}\omega(s)\frac{\mathrm{d}s}{s}\right)\\
				&\leq d\left(\|u\|_{L^\infty(B_1)}+\int_d^{2r_0}[\omega_g+\omega_f](s)\frac{\mathrm{d}s}{s}\right)\exp \left(C\int_d^{2r_0}\omega(s)\frac{\mathrm{d}s}{s}\right),
		  \end{align*}
            and then
            $$|u(x) - u(y)| \leq d\left(\|u\|_{L^\infty(B_1)}+\int_d^{2r_0}[\omega_g+\omega_f](s)\frac{\mathrm{d}s}{s}\right)\exp \left(C\int_d^{2r_0}\omega(s)\frac{\mathrm{d}s}{s}\right).$$
            
	\end{itemize}
        Combining the three possibilities we attain the desired result.
	\end{proof}

        \textbf{Data availability statement.}
Data sharing not applicable to this article as no datasets were generated or analysed during the current study.

    \textbf{Conflict of interest statement.}
The author has no competing interests to declare that are relevant to the content of this article.

\end{document}